\documentclass{prepdmuc}

\newtheorem{theorem}{Theorem}
\newtheorem{corollary}[theorem]{Corollary}
\newtheorem{definition}[theorem]{Definition}
\newtheorem{lemma}[theorem]{Lemma}
\newtheorem{proposition}[theorem]{Proposition}
\newtheorem{remark}[theorem]{Remark}
\newtheorem{example}[theorem]{Example}

\newcommand{\bF}{\overline{F}}
\newcommand{\bG}{\overline{G}}

\newcommand{\IFR}{{\rm c}}
\newcommand{\DFR}{{\rm DFR}}

\newcommand{\IFRA}{\ensuremath{\ast}}
\newcommand{\DFRA}{{\rm DFRA}}

\newcommand{\dI}{\ensuremath{\mathbb{I}}}

\title[Non comparability with respect to the convex transform order]{Non comparability with respect to the convex transform order with applications\thanks{This work was partially supported by the Centre for Mathematics of the University of Coimbra -- UID/MAT/00324/2013, funded by the Portuguese Government through FCT/MEC and co-funded by the European Regional Development Fund through the Partnership Agreement PT2020.}}

\autor[I. Arab]{Idir Arab}{CMUC, Department of Mathematics, University of Coimbra, EC Santa Cruz, 3001-501 Coimbra, Portugal.}{}
\email{idir.bhh@gmail.com}

\autor[M. Hadjikyriakou]{Milto Hadjikyriakou}{University of Central Lancashire, Cyprus.}{}
\email{MHadjikyriakou@uclan.ac.uk}

\autor[P.E. Oliveira]{Paulo Eduardo Oliveira}{CMUC, Department of Mathematics, University of Coimbra, EC Santa Cruz, 3001-501 Coimbra, Portugal.}{This work was partially supported by the Centre for Mathematics of the University of Coimbra -- UID/MAT/00324/2013, funded by the Portuguese Government through FCT/MEC and co-funded by the European Regional Development Fund through the Partnership Agreement PT2020.}
\email{paulo@mat.uc.pt}

\keywords{convex transform order; failure rate; rapidly varying functions.}

\subjclass{60E15, 60E05, 62N05}

\date{July, 2019}

\infonum{19}{27}

\date{}

\begin{document}

\maketitle

\begin{abstract}
In the literature of stochastic orders, one rarely finds results that can be considered as criteria for the non-comparability of random variables. In this paper, we provide results that enable researchers to use simple tools to conclude that two random variables are not comparable with respect to the convex transform order. The criteria are applied to prove the non-comparability of parallel systems with components that are either exponential, Weibull or  Gamma distributed, providing a negative answer for a conjecture about comparability with respect to the convex transform order in a much broader scope than its initial statement.
\end{abstract}

\section{Introduction}
Failure rates functions provide a natural way to describe ageing properties of lifetime distributions. The monotonicity of the failure rate studied, among others, by Barlow and Proschan~\cite{BP-81}, Patel~\cite{Pt-83}, Sengupta~\cite{Sp-94}, El-Bassiouny~\cite{ElB-03} or, for a more recent account of results, by Shaked and Shanthikumar~\cite{SS07} or Marshall and Olkin~\cite{MO07}, is amidst the first relevant properties of interest as this gives an easily interpretable evolution of the trend of the ageing process. The comparison of this evolution for complex systems whose lifetimes are random arises naturally, once we have defined a criterium to measure this ageing. Being interested in the properties of the failure rate, the comparison of different systems should be based on the long run behaviour of the corresponding failure rate functions. Such comparisons of ageing properties may be thought of as orderings between the lifetime distributions, for which there are several alternative notions available in the literature. 

The specific order relation we will be considering was introduced by van Zwet~\cite{VZ-64} and is known either as the convex transform order (as in Shaked and Shanthikumar~\cite{SS07} or Kochar and Xu~\cite{KX09,KX11}) or as the increasing failure rate order (as in Averous and Meste~\cite{AM89}, Fagiuoli and Pellerey~\cite{FP93}, Nanda et al.~\cite{ASOK}, Arab and Oliveira~\cite{AO18,AO18a} or Arab et al.~\cite{AHO18,AHO19}). As expressed by Definition~\ref{DEF S-IFR} below, the convex transform order is defined through the relative convexity between the quantile functions. This notion of relative convexity between two functions has been introduced in Hardy et al.~\cite{HarLittPol59} and is, in general, characterized by properties of the quotient between the derivatives of the functions (see Rajba~\cite{Raj14} for some recent results). Computationally this is a difficult relationship to be established as in most cases the quantile functions do not have closed representations. However, the convex transform order may be adequately expressed through a control on the number of intersection points of the graphical representations of the distribution functions 
subject to the effect of zooming or shifting on the elapsing of time. This approach has been explored in Arab and Oliveira~\cite{AO18,AO18a} and Arab et al.~\cite{AHO18,AHO19} to derive explicit order relations within the Gamma and Weibull families of distributions and between these two families. Applications to ageing comparisons of parallel systems have been discussed by Kochar and Xu~\cite{KX09,KX11} and Arab et al.~\cite{AHO18,AHO19}. The lifetime of parallel systems, described by the maximum of the lifetimes of each component, means that we need to consider distribution functions for which either it is not possible to obtain the inverse or this inverse is very difficult to be studied, even in the case of exponentially distributed components. The results that are the motivation for this work, stem from the comparison between parallel systems. The ordering with respect to the convex transform order of homogeneous parallel systems of exponentially distributed components with non-homogenous ones was proved by Kochar and Xu~\cite{KX09,KX11} showing that homogeneous parallel systems age faster than nonhomogeneous systems. Using a completely different approach, Arab et al.~\cite{AHO18} extended this ordering relationship to a broader framework. The comparison of parallel systems with exponentially distributed components, each one with a different hazard rate, was discussed in Kochar and Xu~\cite{KX09}, conjecturing, based on numerical evidence, that the same convex transform order should hold under suitable relationships between the hazard rates of the components. 
The conjectured ordering was proved to be false by Arab et al.~\cite{AHO19}, were it is shown that parallel systems with two components, with hazard rates satisfying the adequate relationship, are non-comparable by providing a general way to build counterexamples. However, as proved by Kochar and Xu~\cite{KX11} the ordering relationship holds with respect to a weaker form of transform order, namely the star transform order. Arab et al.~\cite{AHO19} provided a simpler proof of this latter ordering.

The results and arguments for the proofs introduced in Arab et al.~\cite{AHO19} suggested looking for the identification of relationships between the distributions functions leading to the non-comparability results to be discussed below. These results are then applied to obtain the non-comparability with respect to the convex transform order of more general parallel systems, allowing to consider systems with arbitrarily large number of components that may even have lifetime distributions that are Gamma or Weibull distributed, which was out of scope of the applicability of the methodology used in Arab et al.~\cite{AHO18,AHO19}.

The paper is structured as follows. In Section~2 we provide some definitions and results that will be useful for the sequel. In Section~3, orderings of parallel systems with homogeneous components are studied while in Section~4 we provide some criteria for identifying non-comparable parallel systems. A subclass of rapidly varying functions is employed in Section~5 in order to get results that will allow us to discuss non-comparability of random variables with respect to the convex transform order and finally, a number of applications is presented in Section~6.

\section{Preliminaries}
Let $X$ be a nonnegative random variable with density function $f_X$, distribution function $F_X$, and tail function $\overline{F}_X=1-F_X$. Moreover, recall that, for each $x\geq 0$, the failure rate function of $X$ is given by $r_{X}(x)=\frac{f_X(x)}{1-F_X(x)}$. Two of the most simple and common ageing notions are defined in terms of the failure rate function. Their definitions are given below. 
\begin{definition}
Let $X$ be a nonnegative valued random variable.
\begin{enumerate}
\item
$X$ is said {\rm IFR} (resp., \DFR) if $r_{X}$ is increasing (resp., decreasing) for $x\geq 0$.
\item
$X$ is said {\rm IFRA} (resp., $\DFRA$) if $\frac{1}{x}\int_0^x r_{X,s}(t)\,dt$ is increasing (resp., decreasing) for $x>0$.
\end{enumerate}
\label{def:s-fail}
\end{definition}
The above definitions refer to failure rate monotonicity properties of the distribution. Next, we recall the transform orders defining the comparison between distribution functions.
\begin{definition}
\label{DEF S-IFR}
Let $\mathcal{F}$ denote the family of distributions functions such that $F(0)=0$. Let $X$ and $Y$ be nonnegative random variables with distribution functions $F_X,F_Y\in\mathcal{F}$.
    \begin{enumerate}
    \item
    The random variable $X$ is said to be smaller than $Y$ in the convex transform order, and we write $X\leq_{\IFR}Y$, if $\overline{F}_Y^{-1}(\overline{F}_X(x))$ is convex.
    \item
    The random variable $X$ is said to be smaller than $Y$  in the star transform order, and we write $X\leq_{\IFRA}Y$, if $\frac{\overline{F}_Y^{-1}(\overline{F}_X(x))}{x}$ is increasing (this is also known as $\overline{F}_Y^{-1}(\overline{F}_X(x))$ being star-shaped).
    \end{enumerate}
\end{definition}
\begin{remark}
Some notational issues and an alternative formulation of the previous definition.
    \begin{enumerate}
    \item
    The convex and star transform orders are also known as {\rm IFR} and {\rm IFRA} orders, respectively (see, for example, Nanda et al.~\cite{ASOK}).
    \item
    $X$ being smaller, in either sense, is often read as $X$ ageing faster than $Y$.
    \item
    As $F_Y^{-1}(F_X(x))=\overline{F}_Y^{-1}(\overline{F}_X(x))$, Definition~\ref{DEF S-IFR} may also be presented referring to the distribution functions instead of the tails.
    \end{enumerate}
\end{remark}

A general characterization of the above transform order relations is given below (see Propositions~3.1 and 4.1 in Nanda et al.~\cite{ASOK}).
\begin{theorem}
\label{convexity-equivalence}
Let $X$ and $Y$ be random variables with distribution functions $F_X,F_Y\in\mathcal{F}$.
    \begin{enumerate}
    \item
    $X\leq_{\IFRA}Y$ if and only if for any real number $a$, $\bF_{Y}(x)-\bF_{X}(ax)$ changes sign at most once, and if the change of signs occurs, it is in the order ``$-,+$'', as $x$ traverses from $0$ to $+\infty$.
    \item
    $X\leq_{\IFR}Y$ if and only if for any real numbers $a$ and $b$, the function $V(x)=\bF_{Y}(x)-\bF_{X}(ax+b)$ changes sign at most twice, and if the change of signs occurs twice, it is in the order ``$+,-,+$'', as $x$ traverses from $0$ to $+\infty$.
    \end{enumerate}
\begin{remark}
\label{rem:imply}
It is obvious from Theorem~\ref{convexity-equivalence} that if $X\leq_{\IFR}Y$, then we also have that $X\leq_{\IFRA}Y$.
\end{remark}
\begin{remark}
\label{AOalpha}
As mentioned in Remark 25 in Arab and Oliveira~\cite{AO18,AO18a}, it is enough to verify the above characterizations for $a>0$.
\end{remark}
\end{theorem}
The actual verification of the sign variation required by Theorem~\ref{convexity-equivalence} is often not directly achievable.
Computationally tractable alternatives with applications were studied in Arab and Oliveira~\cite{AO18,AO18a} and Arab et al.~\cite{AHO18} (see, Theorems~2.3 and 2.4 in the latter reference). As stated in Theorem~\ref{thm:newcrit} below, establishing first the star transform ordering simplifies the condition to be verified for the convex transform order. 
\begin{theorem}[Theorem~29 in Arab et al.~\cite{AHO18}]
\label{thm:newcrit}
Let $X$ and $Y$ be random variables with distribution functions $F_X,F_Y\in\mathcal{F}$, respectively. If $X\leq_{\IFRA}Y$ and the criterium in (\textit{2}) from Theorem~\ref{convexity-equivalence} is verified for $b\geq0$, then $X\leq_{\IFR}Y$.
\end{theorem}

The convex and the star transform orders fall in the family of iterated {\rm IFR} and {\rm IFRA} orders, respectively, introduced and studied in Nanda et al.~\cite{ASOK}, Arab and Oliveira~\cite{AO18,AO18a} or Arab et al.~\cite{AHO18}. It is useful to recall that Nanda et al.~\cite{ASOK} proved that the iterated {\rm IFR} and {\rm IFRA} orderings define partial order relations in the equivalence classes of $\mathcal{F}$ corresponding to the equivalence relation $F\sim G$ defined by $F(x)=G(kx)$, for some $k>0$. In case of families of distributions that have a scale parameter, this allows to choose the parameter in the most convenient way. 

Finally, we introduce some notation to be used throughout our results.
\begin{definition}
Let $X_1,\ldots,X_n$ be a sample of the random variable $X$. The order statistics of the sample are denoted by $X_{1:n}\leq X_{2:n}\leq\cdots\leq X_{n:n}$.
\end{definition}

\section{Parallel systems with homogeneous components}
Comparing the ageing properties of parallel systems based on homogeneous components with exponentially distributed lifetimes was a byproduct of the results proved in the final section of Arab et al.~\cite{AHO18}. Indeed, Corollary~40 in \cite{AHO18} proves that such parallel systems age faster with respect to the convex transform order as the number of components increases. This comparability of the ageing behaviour with respect to each of the transform orders defined may show a different pattern depending on the components lifetime. The following results describe the ordering relations for the case of Weibull distributed components, where, opposite to the result proved in Corollary~40 in Arab et al.~\cite{AHO18}, a different behaviour with respect to each of the orders is found. Moreover, this extends Theorem~3.1 in Kochar and Xu~\cite{KX11}, going beyond exponentially distributed components.
\begin{proposition}
\label{example_Max_WEIBULL}
Let $X_1,\ldots,X_n$ 
be independent and identically distributed Weibull random variables with distribution functions $F(x)=1-e^{-x^\alpha}$, for $x\geq 0$, where $\alpha>0$.
Given integers $m\leq k\leq n$, we have $X_{k:k}\leq_{\IFRA}X_{m:m}$
\end{proposition}
\begin{proof}
Denote by $F_k(x)=F^k(x)$ and $F_m(x)=F^m(x)$ the distribution functions of $X_{k:k}$ and $X_{m:m}$, respectively. We need to prove that
$$
C_{k,m}(x)=F_m^{-1}(F_k(x))=\left(-\ln\left(1-\left(1-e^{-x^{\alpha}}\right)^{\frac{k}{m}}\right)\right)^{\frac{1}{\alpha}}
$$
is star-shaped. This is equivalent to verifying that $H(x)=F_k(x)-F_m(ax)$ changes sign at most once in the order ``$-,+$'', when $x$ goes from 0 to $+\infty$, for every $a>0$. In fact, as $F_m(x)> F_k(x)$, it is enough to consider $0<a<1$. As the logarithmic function is an increasing transformation, $H(x)$ has the same sign variation as
$$
H_1(x)=k\log(1-e^{-x^\alpha})-m\log(1-e^{-a^\alpha x^\alpha}).
$$
It is easily verified that $\lim_{x\rightarrow 0}H_1(x)=-\infty$ and $\lim_{x\rightarrow+\infty}H_1(x)=k-m>0$. To characterize the behaviour of $H_1$ we look at its derivative. After simple algebraic manipulation, the sign of $H_1^\prime$ is easily seen to be the same as the sign of
$$
H_1^*(x)=-ma^\alpha e^{-a^\alpha x^\alpha}+ke^{-x^\alpha}+ (ma^\alpha-k)e^{-(a^\alpha+1)x^\alpha}.
$$
This is a polynomial of exponentials. Taking into account the results in Tossavainen~\cite{Toss07}, this polynomial, hence $H_1^\prime$, has at most two real roots and its sign variation is either ``$-$'' or ``$+,-$''. Consequently, the sign variation of $H_1$ is at most ``$-,+$'', so the proof is concluded.
\end{proof}
As what regards the convex transform ordering between $X_{k:k}$ and $X_{m:m}$, the following example shows that, in general, they are not comparable.
\begin{example}
We prove that $C_{k,m}(x)$ is, in general, neither convex nor concave. For this purpose, choose $k=2m$ and $\alpha=2$. Then, we may compute explicitly to find
$$
C_{2m,m}^{\prime}(x)=\frac{2x\left(e^{x^2}-1\right)}{\left(2e^{x^2}-1\right)\sqrt{-\ln\left(2e^{-x^2}-e^{-2x^2}\right)}}.
$$
A direct verification shows that $\lim_{x\rightarrow 0}C_{2m,m}^{\prime}(x)=0$, $\lim_{x\rightarrow +\infty}C_{2m,m}^{\prime}(x)=1$, and $C_{2m,m}^{\prime}(1)\approx 1.08453$, which means that $C_{2m,m}^{\prime}(x)$ is not monotone, hence $C_{2m,m}$ is neither convex nor concave. A more general statement about the non-comparability with respect to the convex transform order is stated later in Proposition~\ref{ex:+examp}.
\end{example}

A convenient adaptation of a monotonicity result for families of functions depending on a single parameter provides another criterium for the star transform order. 

First, recall the following monotonicity characterization.
\begin{theorem}[Saunders and Moran~\cite{SM78}]
\label{thm:IFRA}
Let $F_\alpha$, where $\alpha$ is a real number, be a family of distribution functions. Then  $\frac{F_\alpha^{-1}(x)}{F_\alpha^{-1}(y)}$ decreases (resp., increases) with respect to $\alpha$, for each fixed $x\geq y$, if and only if $D(\alpha,x)=\frac{F_\alpha^{\prime}(x)}{xf_\alpha(x)}$ increases (resp., decreases) with respect to $x$, where $F_\alpha^{\prime}(x)=\frac{\partial}{\partial\alpha}F_\alpha(x)$ and $f_\alpha(x)=\frac{\partial}{\partial x}F_\alpha(x)$.
\end{theorem}
\begin{lemma}
\label{lem:IFRA}
Let $F_\alpha$, where $\alpha$ is a real number, be a family of distribution functions. Then $F_\alpha$ is increasing (resp., decreasing) in the star transform order with respect to $\alpha$, (that is, if $\alpha_1\leq \alpha_2$ then $F_{\alpha_1}\leq_{\IFRA}F_{\alpha_2}$) if and only if, with the notations of Theorem~\ref{thm:IFRA},  $D(\alpha,x)=\frac{F_\alpha^{\prime}(x)}{xf_\alpha(x)}$ is decreasing (resp., increasing) with respect to $x$.
\end{lemma}

\begin{proof}
Consider two real numbers $\alpha_1\leq\alpha_2$. If $F_{\alpha_1}\leq_{\IFRA}F_{\alpha_2}$, then, due to the increasingness of $F_{\alpha_1}^{-1}$, it follows that, for $y\leq x$, we have
$$
\frac{F^{-1}_{\alpha_2}}{F^{-1}_{\alpha_1}}(y)\leq \frac{F^{-1}_{\alpha_2}}{F^{-1}_{\alpha_1}}(x)
\quad\Leftrightarrow\quad \frac{F^{-1}_{\alpha_2}(x)}{F^{-1}_{\alpha_2}(y)}\geq \frac{F^{-1}_{\alpha_1}(x)}{F^{-1}_{\alpha_1}(y)},
$$
which is equivalent to $\frac{F^{-1}_\alpha(x)}{F^{-1}_\alpha(y)}$ being increasing with respect to $\alpha$, hence, using Theorem~\ref{thm:IFRA}, still equivalent to $\frac{F_\alpha^{\prime}(x)}{xf_\alpha(x)}$ being decreasing with respect to $x$.
\end{proof}
The previous lemma enables a few more comparisons with respect to the star transform order. We present next an example showing that even with respect to the star transform order, the comparison between parallel systems may fail to hold.
\begin{example}
Let $X_1,\ldots,X_n$ 
be independent and identically distributed random variables with $u$-quadratic distribution with density $f(x)=\frac{3}{16}(x-2)^2\dI_{[0,4]}(x)$,
and distribution function
$F(x)=\frac{1}{16}\left((x-2)^3+8)\right)$, for $0\leq x\leq4$.
Given distinct integers $k$ and $m$ to decide about the star transform order between $X_{k:k}$ and $X_{m:m}$, according to Lemma~\ref{lem:IFRA}, it is enough to look at the monotonicity with respect to $x$ of $D(\alpha,x)=\frac{F_\alpha^{\prime}(x)}{xf_\alpha(x)}$, where $F_\alpha(x)=F^\alpha(x)$ and $F_\alpha^\prime(x)=\frac{\partial}{\partial\alpha}F^\alpha(x)$.
Then $D(\alpha,x)=\frac{F(x)\log(F(x))}{xf(x)}$ and it is easy to check that, for each $\alpha>0$ fixed, $\lim_{x\rightarrow0}D(\alpha,x)=-\infty$, $\lim_{x\rightarrow 2}D(\alpha,x)=-\infty$, and $D(\alpha,4)=0$, which means that $D(\alpha,x)$ is, for each $\alpha>0$ fixed, not monotone with respect to $x$, hence $X_{k:k}$ and $X_{m:m}$ are not comparable with respect to the star transform order. Of course, taking into account Remark~\ref{rem:imply}, these variables are also not comparable with respect to the convex transform order.
\end{example}
We complete this section with some positive results about the comparison of parallel systems with respect to the star transform order applying Theorem~\ref{thm:IFRA}. Recall that a nonnegative random variable is said to have the generalized exponential distribution, denoted by $GE(\alpha,\lambda)$ if its distribution function is $F_{X}(x) = (1-e^{-\lambda x})^\alpha$, for $x\geq 0$, where $\alpha,\lambda>0$ are the shape and scale parameters, respectively.
\begin{proposition}
Let $X_1,\ldots,X_n$ 
be independent random variables with $GE(\alpha,\lambda)$ distributions, and $Y_1,\ldots,Y_n$ 
be independent random variables with $GE(\alpha^\ast,\lambda^\ast)$ distributions. 
\begin{enumerate}
\item
Given $m>k$, it holds that $X_{m:m}\leq_{\IFRA}X_{k:k}$.
\item
If $\alpha>\alpha^\ast$ and $\lambda=\lambda^\ast$, then for every $n\geq 1$, $X_{n:n}\leq_{\IFRA}Y_{n:n}$.
\item
If $\alpha=\alpha^\ast$, then given integers $m$ and $k$, $X_{k:k}$ and $X_{m:m}$ are equivalent with respect to the star transform order.
\end{enumerate}
\end{proposition}
\begin{proof}
The proof reduces to identifying the function $D(\cdot,x)$ corresponding to the adequate parameter for the usage of Theorem~\ref{thm:IFRA}, and verifying that it satisfies the appropriate monotonicity.
\begin{enumerate}
\item
For this case we take the parameter to be $n$, considering $F_n(x)=F_{X_{n:n}}(x)=(1-e^{-\lambda x})^{\alpha n}$ ($\alpha$ and $\lambda$ are considered constant). Then
$$
D(n,x)=\frac{F_n^\prime(x)}{xf_n(x)}=\frac{1}{n\lambda}\left(\frac{e^{\lambda x}-1}{x}\right)\log (1-e^{-\lambda x})
$$
and
$$
\frac{\partial D}{\partial x}(n,x) =\frac{((\lambda x-1)e^{\lambda x}+1)\log(1-e^{-\lambda x})+\lambda x}{x^2\lambda n}. 
$$
Some elementary calculus arguments show that $\frac{\partial D}{\partial x}(n,x)>0$, for every $x>0$, thus, taking into account Theorem~\ref{thm:IFRA}, $F_n(x)$ decreases in star transform order with respect to $n$, that is, for $m>k$, $X_{m:m}\leq_{\IFRA}X_{k:k}$.
		
\item
We now consider $\alpha$ as the parameter: $F_{\alpha}(x) = F_{X_{n:n}}(x)$ ($n$ and $\lambda$ constant). It is easily verified that $F_{\alpha}(x)$ decreases with respect to $\alpha$, so the conclusion follows.
		
\item
Considering $n$ and $\alpha$ constant, it follows that $D(\lambda,x) = \frac{1}{\lambda}$, meaning the $X_{m:m}$ and $X_{k:k}$ are equivalent with respect to the star transform order.
\end{enumerate}
\end{proof}

%

\section{Non-comparability of parallel systems}
\label{sec:nonconvex}
The arguments used in Arab et al.~\cite{AHO19} to prove the non-comparability of systems suggested we should look at the function $V$ introduced in Theorem~\ref{convexity-equivalence} as function of three variables $(x,a,b)$, instead of the usual approach that fixed the choices for $a$ and $b$, and then studied the sign variation with respect to $x$. The proof of Theorem~11 in Arab et al.~\cite{AHO19} shows that exploring the general behaviour based on all the three variables provides the information needed for concluding the non-comparability of the parallel systems treated there. As a result of a few numerical simulations, this approach suggested that whenever $\overline{F}_Y^{-1}(\overline{F}_X(x))$ is asymptotically linear, then non-comparability should hold. In the sequel we prove this assertion, provide some applications to the comparability between parallel systems, and, in particular, respond negatively to the conjecture raised by Kochar and Xu~\cite{KX09} with a weaker assumption on the hazard rates for the distributions involved.

The next result provides simple characterization of non-convexity or non-concavity.
\begin{theorem}
\label{CONV_Characterization}
Let $h$ be an increasing function defined on the positive real line such that $h(0)=0$, and assume that there exist real numbers $b$ and $c$ such that $\lim_{x\rightarrow+\infty}\bigl(h(x)-(bx+c)\bigr)=0$.
\begin{enumerate}
\item
If $c> 0$, then $h$ is not convex.
\item
If $c<0$, then $h$ is not concave.
\item
If $c=0$ and $h(x)\neq bx$, then $h$ is neither convex nor concave.
\end{enumerate}
\end{theorem}
\begin{proof}
As $\lim_{x\rightarrow+\infty}\bigl(h(x)-(bx+c)\bigr)=0$, it follows that $\lim_{x\rightarrow+\infty}h^{\prime}(x)=b$. Define now $g(x)=h(x)-(bx+c)$. Assume $c>0$ and that $h$ is convex. Then, it follows that $g$ is decreasing, which is impossible as $g(0)=-c<0$ and $\lim_{x\rightarrow+\infty}g(x)=0$. Likewise, if we assume that $c\leq0$ and $h$ is concave, then it follows that $h$ is increasing, which is not compatible with the fact that $g(0)=-c>0$ and $\lim_{x\rightarrow+\infty}g(x)=0$. Finally, if $c=0$, then $h$ cannot be monotonic as $g(0)=0$ and $\lim_{x\rightarrow+\infty}g(x)=0$. Consequently, $h^{\prime}$ is not monotonic.
\end{proof}
We may now complete the description of ordering relationship between the random variables considered in Proposition~\ref{example_Max_WEIBULL}.
\begin{proposition}
\label{ex:+examp}
Let $X_1,\ldots,X_n$ 
be independent and identically distributed Weibull random variables with distribution functions $F(x)=1-e^{-x^\alpha}$, for $x\geq 0$, with $\alpha>1$.
Given distinct integers $m,k\leq n$, the random variables $X_{k:k}$ and $X_{m:m}$ are not comparable with respect to the convex transform order.
\end{proposition}
\begin{proof}
With the notation introduced in the course of the proof of Proposition~\ref{example_Max_WEIBULL}, it follows that $\lim_{x\rightarrow+\infty}\bigl(C_{k,m}(x)-x\bigr)=0$. Therefore, according to Theorem~\ref{CONV_Characterization}, as $C_{k,m}$ is obviously not linear, it is neither convex nor concave, so $X_{k:k}$ and $X_{m:m}$ are not comparable with respect to the convex transform order.
\end{proof}
\begin{remark}
\label{rmk:exp-weibull}
The limit $\lim_{x\rightarrow+\infty}\bigl(C_{k,m}(x)-x\bigr)=0$ mentioned above holds even if $k$ and $m$ are not integers. Consequently, the non comparability follows between random variables with exponentiated Weibull distributions (see Mudholkar and Srivastava~\cite{MS93} for properties and applications of this family of distributions).
\end{remark}

\begin{example}
We now consider an extension of Proposition~\ref{example_Max_WEIBULL}. Let $X_1,\ldots,X_n$ 
be independent and identically distributed Weibull random variables with shape parameter $\alpha>0$ and scale parameter $\lambda>0$.
Given integers $m,k\leq n$, consider the random variables $X_{m:m}$ and $X_{1:k}$. 
Denoting by $F_m$ the distribution function of $X_{m:m}$, as before, and by $G_k$ the distribution function of $X_{1:k}$ we may compute explicitly
$$ F_m^{-1}(G_k(x))=\frac{1}{\lambda}\left(-\ln\left(1-\left(1-e^{-\left(\lambda k^{1/\alpha}x\right)^\alpha}\right)^{\frac{1}{m}}\right)\right)^{\frac{1}{\alpha}},
$$
and verify that
$\lim_{x\rightarrow+\infty}\bigl(F_m^{-1}(G_k(x))-\lambda k^{1/\alpha}x\bigr)=0$. Therefore, according to Theorem~\ref{CONV_Characterization}, $X_{m:m}$ and $X_{1:k}$ are not comparable with respect to the convex transform order.
\end{example}

Using Theorem~\ref{CONV_Characterization} to conclude about the non-comparability of random variables $X$ and $Y$ requires being able to characterize the asymptotic behaviour of $\overline{F}_Y^{-1}(\overline{F}_X(x))$. These functions are often not available, so it is preferable to have a criterium based on the comparison of the density functions.
\begin{theorem}
\label{thm:comp1}
Let $F$ and $G$ be distribution functions of class $\mathcal{F}$ with densities $f$ and $g$, respectively. Assume there exists $c>0$ such that
\begin{equation}\label{NON-CVX}
\forall\varepsilon>0, \exists A>0,\, x \geq A \;\Rightarrow\;
cg(cx+\varepsilon)\leq f(x)\leq cg(cx-\varepsilon).
\end{equation}
Then $G^{-1}(F(x))$ is neither convex nor concave or else $G^{-1}(F(x))=cx$.
\end{theorem}
\begin{proof}
Assume (\ref{NON-CVX}) holds, choose an arbitrary $\varepsilon>0$, and define
$$
H_1(x)=F(x)-G(cx+\varepsilon),\quad\mbox{and}\quad H_2(x)=F(x)-G(cx-\varepsilon).
$$
As $\lim_{x\rightarrow+\infty}H_1(x)=0$ and $H_1^\prime(x)=f(x)-cg(cx+\varepsilon)\geq 0$, for every $x\geq A$, it follows that $F(x)\leq G(cx+\varepsilon)$. Analogously, using now $H_2$, it also follows that $G(cx-\varepsilon)\leq F(x)$. Consequently, when $x\geq A$, we have that $G(cx-\varepsilon)\leq F(x)\leq G(cx+\varepsilon)$,
which implies $\left\vert G^{-1}(F(x))-cx\right\vert\leq \varepsilon$, therefore, $\lim_{x\rightarrow +\infty}\bigl(G^{-1}(F(x))-cx\bigr)=0$. Taking now into account Theorem~\ref{CONV_Characterization}, the conclusion follows.
\end{proof}

As mentioned earlier, Arab et al.~\cite{AHO19} proved that parallel systems with two components with exponentially distributed lifetimes are not comparable with respect to the convex transform order. As suggested by the conjecture by Kochar and Xu~\cite{KX09}, the hazard rates must satisfy a specific relation. We recall here the complete result.
\begin{proposition}[Theorem~11 in Arab et al.~\cite{AHO19}]
\label{prop:non-kochar}
Let $X_1$, $X_2$, $Y_1$ and $Y_2$ be exponential random variables with hazard rates $\lambda_1$, $\lambda_2$, $\theta_1$ and $\theta_2$ satisfying, $0<\lambda_1<\lambda_2$,  $0<\theta_1<\theta_2$, $\lambda_1<\theta_1$ and $\lambda_1+\lambda_2=\theta_1+\theta_2$, and assume $X_1$ is independent from $X_2$ and $Y_1$ is independent from $Y_2$. Then the variables $X_{2:2}=\max(X_1,X_2)$ and $Y_{2:2}=\max(Y_1,Y_2)$ are not comparable with respect to the convex transform order.
\end{proposition}
\begin{remark}
Using the fact that the convex transform order is invariant with respect to the multiplication of the random variables by constants, it is easily seen that the conclusion in Proposition~\ref{prop:non-kochar} still holds when \textit{(i)} the hazard rates of $X_1$ and $X_2$ are 1 and $\lambda\ne 1$, respectively, and the hazard rates of $Y_1$ and $Y_2$ are 1 and $\theta\ne1$, respectively, assuming that $\lambda\ne\theta$, or \textit{(ii)} the hazard rates of $X_1$ and $X_2$ are $\lambda_1$ and $\lambda_2$, respectively, and the hazard rates of $Y_1$ and $Y_2$ are $\theta_1$ and $\theta_2$, respectively, assuming that $\frac{\lambda_2}{\lambda_1}\geq\frac{\theta_2}{\theta_1}$. Hence, the negative answer to the validity of the Kochar and Xu~\cite{KX09} conjecture proved by Arab et al.~\cite{AHO19} holds for a broader choice of parameters than described in Proposition~\ref{prop:non-kochar}. However, the proof given in \cite{AHO19} only applies to parallel systems with two components in each system. 
Using Theorem~\ref{thm:comp1}, we may prove a more general non-comparability result.
\end{remark}

\begin{theorem}
\label{thm:gen-kochar}
Let $X_1,\ldots,X_n$ be independent exponential random variables with hazard rates $\lambda_1<\lambda_2\leq\lambda_3\leq\cdots\leq\lambda_n$, and $Y_1,\ldots,Y_n$ be independent exponential random variables with hazard rates $\theta_1<\theta_2\leq\theta_3\leq\cdots\leq\theta_n$.
Then, for every integers $m$ and $k$, $X_{k:k}$ and $Y_{m:m}$ are not comparable with respect to the convex transform order.
\end{theorem}
\begin{proof}
Denote by $F_k(x)=\prod_{\ell=1}^{k}(1-e^{-\lambda_\ell x})$ and $G_m(x)=\prod_{\ell=1}^{m}(1-e^{-\theta_\ell x})$ the distribution functions of $X_{k:k}$ and $Y_{m:m}$, respectively. The corresponding density functions satisfy $f_k(x)=\lambda_1e^{-\lambda_1x}+P_k(x)$ and $g_m(x)=\theta_1e^{-\theta_1x}+L_m(x)$, where $P_k(x)=o(e^{-\lambda_1x})$ and $L_m(x)=o(e^{-\theta_1x})$. Therefore, choosing $c=\frac{\lambda_1}{\theta_1}$, the densities $f_k$ and $g_m$ satisfy (\ref{NON-CVX}). As $G_m^{-1}(F_k(x))$ is obviously not linear it follows, taking into account Theorem~\ref{thm:comp1}, that $G_m^{-1}(F_k(x))$ is neither convex nor concave, thus concluding the proof.
\end{proof}

\section{Rapid variation and non-comparability}
Although distribution functions, or the corresponding tail functions, are often not available, in some cases it is still possible to compare their behaviour at infinity. This idea is already present in Theorem~\ref{thm:comp1} but the way it is expressed may be inconvenient. An example would be taking the components of the parallel system to have lifetimes that are Gamma distributed. In this section we will prove that, for suitable fast decreasing tails, we may provide alternative criteria for the non-comparability with respect to the convex transform order only based on the comparison of decreasing rates of the tail functions. We start by some definitions and preliminary results concerning the characterization of the decrease rates.
\begin{definition}
A measurable function $f\!:\!(0,\infty)\longrightarrow(0,\infty)$ is said to be regularly varying of index $\alpha \in \mathbb{R}$, denoted $f \in \mathcal{RV}(\alpha)$, if
$$
\forall\lambda>0,\;\lim_{x\rightarrow+\infty}\frac{f(\lambda x)}{f(x)}=\lambda^\alpha.
$$
If $\alpha=0$, then $f$ is said to be slowly varying. Moreover, the function $f$ is said to be rapidly varying of index $+\infty$, denoted $f\in \mathcal{RPV}(+\infty)$, if
$$
 \lim_{x\rightarrow+\infty}\frac{f(\lambda x)}{f(x)}=\begin{cases}
    +\infty, & \text{if $\lambda >1$}\\
    0, & \text{if $\lambda <1$}
  \end{cases}
$$
Finally, $f$ is said to be rapidly varying of index $-\infty$, denoted $f\in \mathcal{RPV}(-\infty)$ if $\frac{1}{f}\in \mathcal{RPV}(+\infty)$
\end{definition}
Note that if $f$ is regularly varying and decreasing, then its index of variation is negative. Next, we introduce similar notions for random variables, referring to the behaviour of the corresponding tail functions.
\begin{definition}
A nonnegative random variable $X$, with tail function $\bF_X$ is said to be regularly, slowly or rapidly varying if $\bF_X$ is regularly, slowly or rapidly varying, respectively.
\end{definition}
The theory of regularly or rapidly varying functions is well established, so we quote here the general characterization of such functions and refer the reader to, for example, Bingham et al.~\cite{BGT89}.
\begin{theorem}[Karamata's Theorem]
\label{Karamata-Theorem}
A positive function $f$ is slowly varying if and only if there exists $B>0$, such that for every $x\geq B$, $f$ can be written in the form
\begin{equation} \label{Karamat-representation}
f(x)=\eta(x)\exp\left( \int_{B}^x\frac{\epsilon(t)}{t}\,dt \right)
\end{equation}
where $\eta(x)>0$ is a measurable function such that $\lim_{x\rightarrow+\infty}\eta(x)$ is finite and positive and $\epsilon(x)$ is a measurable function such that $\lim_{x\rightarrow+\infty}\epsilon(x)=0$.
\end{theorem}

Next, we present some simple results useful in the sequel.
\begin{lemma}\label{lemma-SV}
Let $f$ be a slowly varying and $a(x)$ be such that $\lim_{x\rightarrow+\infty}\frac{a(x)}{x}=0$. Then for every $\lambda>0$ we have
$$
\lim_{x\rightarrow+\infty}\frac{f(\lambda x+a(x))}{f(x)}=1.
$$
\end{lemma}
\begin{proof}
Using Karamata's Theorem, we have that
$$
\lim_{x\rightarrow+\infty}\frac{f(\lambda x+a(x))}{f(x)}
=\lim_{x\rightarrow+\infty}\frac{\eta(\lambda x+a(x))}{\eta(x)}
   \exp\left(\int_{B}^{\lambda x+a(x)}\frac{\epsilon(t)}{t}\,dt -\int_{B}^x\frac{\epsilon(t)}{t}\,dt\right).
$$
To prove the result, it is enough to prove that the argument of the exponential above goes to zero when $x\longrightarrow+\infty$. As $\lim_{x\rightarrow+\infty}\epsilon(x)=0$, given any $\delta>0$, there exists $A>0$ such that for every $x>A$, $\left\vert\epsilon(x)\right\vert<\delta$. Hence
\begin{equation}\label{zero}
T(x)=\left\vert\int_{B}^{\lambda x+a(x)}\frac{\epsilon(t)}{t}\,dt -\int_{B}^x\frac{\epsilon(t)}{t}\,dt\right\vert\leq \delta \int_{I}\frac{1}{t}\,dt,
\end{equation}
where $I=[x,\lambda x+a(x)]$ if $\lambda>1$, and $[\lambda x+a(x), x]$ if $\lambda\leq1$. Now, if $\lambda>1$, it follows that $T(x)\leq\delta\log\frac{\lambda x+a(x)}{x}$, so $\lim_{x\rightarrow+\infty}T(x)\leq\delta\log\lambda$. If we now assume that $\lambda\leq 1$, it follows that, for $x$ large enough, $T(x)\leq\delta\log\frac{x}{\lambda x+a(x)}$, hence $\lim_{x\rightarrow+\infty}T(x)\leq\delta\log\frac1\lambda$. As $\delta>0$ is arbitrary, it follows that $\lim_{x\rightarrow+\infty}T(x)=0$, so the proof is concluded.
\end{proof}

\begin{corollary}\label{corollary-RV}
Let $f\in\mathcal{RV}(\alpha)$, $a(x)$ be such that $\lim_{x\rightarrow+\infty}\frac{a(x)}{x}=0$, and $b, c>0$. Then
$$
\lim_{x\rightarrow+\infty}\frac{f(bx+a(x))}{f(cx)}=\left(\frac{b}{c}\right)^\alpha
$$
where $\alpha$ is the index of regular variation of $f$ and $\lim_{x\rightarrow+\infty}\frac{a(x)}{x}=0$.
\end{corollary}
\begin{proof}
As $f\in \mathcal{RV}(\alpha)$, there exists a slowly varying function $L(x)$ such that $f(x)=x^\alpha L(x)$, so the conclusion follows immediately from the definition of slow variation and Lemma~\ref{lemma-SV}.
\end{proof}

\begin{theorem}\label{EQ-RV}
Let $f\in\mathcal{RPV}(+\infty)$ be monotone, and assume $\phi$ and $\psi$ real functions such that $\liminf_{x\rightarrow+\infty}\frac{\phi(x)}{\psi(x)}>1$ and $\lim_{x\rightarrow+\infty}\psi(x)=+\infty$. Then $\lim_{x\rightarrow+\infty}\frac{f(\phi(x))}{f(\psi(x))}=+\infty$.
\end{theorem}
\begin{proof}
For every $\varepsilon>0$, and for $x$ large enough we have that $(1-\varepsilon)\psi(x)\leq\phi(x)\leq(1+\varepsilon)\psi(x)$. The conclusion follows immediately using the monotonicity of $f$.
\end{proof}
In order to prove non-comparability results with respect to the convex transform order, assuming the tail functions are of class $\mathcal{RPV}(-\infty)$ seems not enough (see Remark~\ref{remark_classes} below for some counterexamples). We need to introduce a suitable subclass of functions.
\begin{definition}
A measurable function $f\!:\!(0,\infty)\longrightarrow(0,\infty)$ is said to be exponentially rapidly varying of index $+\infty$, denoted $f\in \mathcal{ERPV}(+\infty)$, if $f\in\mathcal{RPV}(+\infty)$ and verifies
$$
\left( \lim_{x\rightarrow+\infty}\frac{\phi(x)}{x}=1,\;\mbox{and}\; \lim_{x\rightarrow+\infty}\phi(x)-x\neq 0\right)\;\Rightarrow\; \lim_{x\rightarrow+\infty}\frac{f(\phi(x))}{f(x)}\neq 1.
$$
The function $f$ is said to be exponentially rapidly varying of index $-\infty$, denoted $f\in \mathcal{RPV}(-\infty)$, if $\frac{1}{f}\in \mathcal{ERPV}(+\infty)$.

As before, the random variable $X$, with tail function $\bF_X$, is said exponentially rapidly varying if $\bF_X\in\mathcal{ERPV}(-\infty)$.
\end{definition}
\begin{remark}
The inclusion $\mathcal{ERPV}(+\infty)\subset\mathcal{RPV}(+\infty)$ is strict: it is easily verified that $f(x)=e^{\log^2(x+1)}$ is in the class $\mathcal{RPV}(+\infty)$, but $f\not\in\mathcal{ERPV}(+\infty)$.
\end{remark}

\begin{theorem}\label{NON-Convexity}
Let $\bF,\bG\in\mathcal{ERPV}(-\infty)$ be two tail functions. If there exists $c>0$, such that $\lim_{x\rightarrow+\infty}\frac{\bF(x)}{\bG(cx)}=1$ then, $\lim_{x\rightarrow +\infty}(\bG^{-1}(\bF(x))-cx)=0$.
\end{theorem}
\begin{proof}
Assume that $\lim_{x\rightarrow+\infty}\bigl(\bG^{-1}(\bF(x))-cx\bigr)\neq0$. Then, for $x$ large enough, we may write $\bG^{-1}(\bF(x))=cx+a(x)$, where $\lim_{x\rightarrow+\infty}a(x)\neq0$. We consider the three following cases:
\begin{enumerate}
\item
$\lim_{x\rightarrow+\infty}\frac{a(x)}{x}=0$. In this case, as $\bG\in\mathcal{ERPV}(-\infty)$, it follows that
$$
\lim_{x\rightarrow+\infty}\frac{\bF(x)}{\bG(cx)}=\lim_{x\rightarrow+\infty} \frac{\bG(cx+a(x))}{\bG(cx)}\neq1,
$$
as $\lim_{x\rightarrow+\infty} \frac{cx+a(x)}{cx}=1$ and $\lim_{x\rightarrow+\infty}(cx+a(x)-cx)=\lim_{x\rightarrow+\infty}a(x)\neq0$.

\item
$\lim_{x\rightarrow+\infty}\frac{a(x)}{x}=b_1>0$. Then   $\lim_{x\rightarrow+\infty}\frac{cx+a(x)}{cx}=\frac{c+b_1}{c}$, and, taking into account Theorem~\ref{EQ-RV}, it follows that
$$
\lim_{x\rightarrow+\infty} \frac{\overline{F}(x)}{\overline{G}(cx)}=\lim_{x\rightarrow+\infty} \frac{\overline{G}(cx+a(x))}{\overline{G}(cx)}=
\begin{cases}
+\infty, & \mbox{if $b_1<0$}\\
    0, & \mbox{if $b_1 >0$}
\end{cases}
$$

\item
$\lim_{x\rightarrow+\infty}\frac{a(x)}{x}=+\infty$. In this case we have $\lim_{x\rightarrow+\infty}\frac{cx+a(x)}{cx}=+\infty$, and,  taking again into account Theorem~\ref{EQ-RV}, it follows that
$$
\lim_{x\rightarrow+\infty}\frac{\bF(x)}{\bG(cx)}=\lim_{x\rightarrow+\infty} \frac{\bG(cx+a(x))}{\bG(cx)}=0.
$$
\end{enumerate}
\end{proof}
\begin{remark} \label{remark_classes}
The assumption $\bG\in\mathcal{ERPV}(-\infty)$ is only used to prove the first case in the proof of Theorem~\ref{NON-Convexity}. However, we cannot relax this assumption to any of the larger classes introduced above. Indeed, it follows from Lemma~\ref{lemma-SV} and Corollary~\ref{corollary-RV}, that if $\bG\in\mathcal{RV}(\alpha)$, for any real $\alpha$, and $\bF(x)=\bG(x+a(x))$ where $\lim_{x\rightarrow+\infty}\frac{a(x)}{x}=0$ and $\lim_{x\rightarrow+\infty} a(x)\neq0$, then we will have $\lim_{x\rightarrow+\infty}\frac{\overline{F}(x)}{\overline{G}(x)}=1$ and $\lim_{x\rightarrow+\infty}\bigl(\bG^{-1}(\bF(x))-bx\bigr)\neq0$ for every $b>0$.

We exhibit next a few concrete examples of the above, and also showing that if we assume only that $\bG\in\mathcal{RPV}(-\infty)$ the conclusion of Theorem~\ref{NON-Convexity} may fail:
\begin{enumerate}
\item
Take $G(x)=\frac{1}{ln(x+1)+1}$ and $\bF_d(x)=\bG(dx+\sqrt{x})$. Then both tails are slowly varying and it is easy to verify that $\lim_{x\rightarrow+\infty}\frac{\overline{F}(x)}{\overline{G}(x)}=1$ and
$$
\lim_{x\rightarrow+\infty}\left(\bG^{-1}(\bF_d(x))-bx\right)=\lim_{x\rightarrow+\infty}\left((d-b)x+\sqrt{x}\right)=
\begin{cases}
+\infty, & \mbox{if $d\geq b$}\\
-\infty, & \mbox{if $d<b$}.
\end{cases}
$$

\item
Choose now $G(x)=\frac{1}{x^2+1}$ and $\bF(x)=\bG(x+\sqrt{x})$, so $\bF,\bG\in\mathcal{RV}(-2)$. Then, as for the previous case, we have that $\lim_{x\rightarrow+\infty}\frac{\bF(x)}{\bG(x)}=1$ and
$$
\lim_{x\rightarrow+\infty}\left(\bG^{-1}(\bF(x))-bx\right)=\lim_{x\rightarrow+\infty}\left((1-b)x+\sqrt{x}\right)=
\begin{cases}
+\infty, & \mbox{if $b\leq1$}\\
-\infty, & \mbox{if $b>1$}.
\end{cases}
$$

\item
Finally, an example showing that even under the assumption $\bG\in\mathcal{RPV}(-\infty)$ the asymptotic approximation to a linear function may fail. Consider $\bG(x)=e^{-\log^2(x+1)}$ and $\bF(x)=\bG(x+\log(x+1))$. It is easily verified that $\bF,\bG\in\mathcal{RPV}(-\infty)$. Moreover, it may also be verified, by simple calculus, that $\lim_{x\rightarrow+\infty}\frac{\bF(x)}{\bG(x)}=1$ and
$$
\lim_{x\rightarrow+\infty}\left(\bG^{-1}(\bF(x))-bx\right)=\lim_{x\rightarrow+\infty}\left((1-b)x+\log(x+1)\right)
=\begin{cases}
+\infty, & \mbox{if $b\leq1$}\\
-\infty, & \mbox{if $b>1$}.
\end{cases}
$$
\end{enumerate}
\end{remark}
The assumption in Theorem~\ref{NON-Convexity} deals with tail functions. A version with an assumption on density functions is immediate.
\begin{corollary}
\label{cor:NON-Convexity}
Let $X$ and $Y$ be two exponentially rapidly varying random variables with corresponding densities $f$ and $g$. If $\lim_{x\rightarrow+\infty}\frac{f(x)}{cg(cx)}=1$, then $X$ and $Y$ are not comparable with respect to the convex transform order.
\end{corollary}

\begin{example}
\label{ex:gen-kochar}
The previous corollary provides an easy alternative proof for the result stated in Theorem~\ref{thm:gen-kochar}. Indeed, with the notation introduced for the proof of Theorem~\ref{thm:gen-kochar}, it follows easily, choosing again $c=\frac{\lambda_1}{\theta_1}$, that $\lim_{x\rightarrow+\infty}\frac{f_k(x)}{cg_m(cx)}=1$, hence $\bG_m^{-1}(\bF_k(x))$ is neither convex nor concave.
\end{example}

\section{Applications}
In this section we present a few results about the non-comparability of parallel systems with components that have distributions other than the exponential.
\begin{proposition}
\label{thm:WEIBULL}
Let $X_1,\ldots,X_n$ be independent random variables with Weibull distributions with the same shape parameter $\alpha>0$ and scale parameters $0<\lambda_1<\lambda_2\leq\cdots\leq\lambda_n$. Analogously, let $Y_1,\ldots,Y_n$ be independent random variables with Weibull distributions with the same shape parameter $\alpha$ and scale parameters $0<\theta_1<\theta_2\leq\cdots\leq\theta_n$.
Then, for every integers $m$ and $k$, $X_{k:k}$ and $Y_{m:m}$ are not comparable with respect to the convex transform order.
\end{proposition}
\begin{proof}
Let $\bF_k(x)=1-\prod_{\ell=1}^{k}(1-e^{-\lambda_\ell^\alpha x^\alpha})$ and $\bG_m(x)=1-\prod_{\ell=1}^{m}(1-e^{-\theta_\ell^\alpha x^\alpha})$ be the tail functions of $X_{k:k}$ and $Y_{m:m}$, respectively. Both tail functions are in the class $\mathcal{ERPV}(-\infty)$. The corresponding densities are represented as $f_k(x)=\alpha x^{\alpha-1}\left( \lambda_1^\alpha e^{-\lambda_1^\alpha x^\alpha}+P_k(x)\right)$ and $g_m(x)=\alpha x^{\alpha-1}\left( \theta_1^\alpha e^{-\theta_1^\alpha x^\alpha}+L_m(x)\right)$, where $P_j(x)=o(e^{-\lambda_1^\alpha x^\alpha})$ and $L_m(x)=o(e^{-\theta_1^\alpha x^\alpha})$. Choosing now $c=\frac{\lambda_1}{\theta_1}$, it follows that
$$
\lim_{x\rightarrow+\infty}\frac{f_k(x)}{cg_m(cx)}=\lim_{x\rightarrow+\infty}\frac{\lambda_1^\alpha e^{-\lambda_1^\alpha x^\alpha}+P_k(x)}{(c\theta_1)^\alpha e^{-(c\theta_1)^\alpha x^\alpha}+L_m(cx)}=1,
$$
hence $\bG_m^{-1}(\bF_k(x))$ is neither convex nor concave, thus, taking into account Corollary~\ref{cor:NON-Convexity}, $X_{k:k}$ and $Y_{m:m}$ are not comparable with respect to the convex transform order.
\end{proof}
A similar non-comparability result also holds for parallel systems with Gamma distributed components with integer shape parameters. We need the following auxiliary lemma to handle this case.
\begin{lemma}
Assume $X$ is $\Gamma(\alpha,1)$ where $\alpha\geq 2$ is an integer. Then the tail function of $X$ is $\bF(x)=e^{-x}\left(1+\sum_{\ell=1}^{\alpha-1}\frac{x^\ell}{\ell!} \right)$.
\end{lemma}
\begin{proposition}
\label{thm:GAMMA}
Let $X_1,\ldots,X_n$ be independent random variables with Gamma distributions with integer shape parameters $0<\alpha_1\leq\cdots\leq\alpha_n$ and scale parameters $0<\lambda_1<\lambda_2\leq\cdots\leq\lambda_n$. Analogously, let $Y_1,\ldots,Y_n$ be independent random variables with Gamma distributions with integer shape parameters $0<\beta_1\leq\cdots\leq\beta_n$ and scale parameters $0<\theta_1<\theta_2\leq\cdots\leq\theta_n$.
Given integers $m$ and $k$, if $\alpha_k=\beta_m$ and $\lambda_k=\theta_m$, then $X_{k:k}$ and $Y_{m:m}$ are not comparable with respect to the convex transform order.
\end{proposition}
\begin{proof}
After identifying the distribution functions of $X_{k:k}$ and $Y_{m:m}$, we may proceed as for the proof of Proposition~\ref{thm:WEIBULL}. The distribution functions of $X_{k:k}$ is given by $F_k(x)=\prod_{\ell=1}^{k}\left(1-e^{-\lambda_\ell x}P_{\alpha_\ell,\lambda_\ell}(x)\right)$, while the distribution function of $Y_{m:m}$ is $G_m(x)=\prod_{\ell=1}^{m}\left(1-e^{-\theta_\ell x}P_{\beta_\ell,\theta_\ell}(x)\right)$,
where $P_{a,b}(x)=1+\sum_{\ell=1}^{a-1}\frac{b^\ell x^\ell}{\ell!}$. One may now easily verify that both tail functions are in the class $\mathcal{ERPV}(-\infty)$ and that, under assumptions on the parameters, the corresponding density functions satisfy $\lim_{x\rightarrow+\infty}\frac{f_k(x)}{g_m(x)}=1$, so, taking into account Corollary~\ref{cor:NON-Convexity}, the conclusion follows.
\end{proof}

\begin{proposition}
\label{prop:gle1}
Let $X_i\sim GE(\alpha_i, \lambda_i)$, $i=1,\ldots,n$, be independent with parameters $\alpha_i,\lambda_i>0$ and $Y_i\sim GE(\beta_i, \theta_i)$, $i=1,\ldots,n$, be independent random variables with parameters $\beta_i,\theta_i>0$. Assume, without loss of generality, that both families of parameters are ordered increasingly.
Given integers $m,k\leq n$, if $\prod_{\ell=1}^k\alpha_\ell=\prod_{\ell=1}^m\beta_\ell$, then the random variables $X_{1:k}$ and $Y_{1:m}$ are not comparable with respect to convex transform order.
\end{proposition}
\begin{proof}
The tail functions of $X_{1:k}$ and $Y_{1:m}$ are $\bF_k(x)=\prod_{\ell=1}^{k}(1-(1-e^{-\lambda_\ell x})^{\alpha_\ell})$ and $\bG_m(x)=\prod_{\ell=1}^{m}(1-(1-e^{-\theta_\ell x})^{\beta_\ell})$, respectively. Expanding and taking into account that the parameters are ordered increasingly, it follows that for $x$ large enough
$\bF_k(x)\simeq\prod_{\ell=1}^{k}\alpha_\ell e^{-\sum_{j=1}^{k}\lambda_jx}$ and $\bG_m(x)\simeq\prod_{\ell=1}^{m}\beta_\ell e^{-\sum_{j=1}^{m}\theta_jx}$. Therefore, $\bF_k,\bG_m\in\mathcal{ERPV}(-\infty)$, so, choosing $c=\frac{\sum_{j=1}^k\lambda_j}{\sum_{j=1}^m\theta_j}$
$$
\lim_{x\rightarrow+\infty}\frac{\bF_k(x)}{\bG_m(cx)}=\frac{\prod_{\ell=1}^{k}\alpha_\ell}{\prod_{\ell=1}^{m}\beta_\ell}=1,
$$	
consequently, taking into account Theorems~\ref{NON-Convexity} and \ref{CONV_Characterization}, the conclusion follows.
\end{proof}
\begin{proposition}
Take $X_1,\ldots,X_n$ and $Y_1,\ldots,Y_n$ 
as in Proposition~\ref{prop:gle1}.
If $\alpha_1=\beta_1$, then the variables $X_{k:k}$ and $Y_{m:m}$ are not comparable with respect to convex transform order.
\end{proposition}
\begin{proof}
Just remark that, as $x\longrightarrow+\infty$,
$\bF_k(x)=1- \prod_{\ell=1}^{k}(1-\alpha_\ell e^{-\lambda_\ell x})=\alpha_1e^{-\lambda_1 x}+o(e^{-\lambda_1 x})$ and $\bG_m(x)=1- \prod_{\ell_1}^m(1-\beta_\ell e^{-\theta_\ell x})=\beta_1e^{-\theta_1 x}+o(e^{-\theta_1 x})$ and choose $c=\frac{\lambda_1}{\theta_1}$.
\end{proof}

Finally, we go back to the case of exponential components, but consider more complex models: first we look at the lifetime of a series system whose components are themselves  parallel systems, and in the second model we consider a case where one of the vectors that are involved has an F-G-M joint distribution (see for example Kotz et al. \cite{KBJ00}).

\begin{proposition}
Let $X_1,\ldots,X_n$ 
be independent exponentially distributed random variables with hazard rates $\lambda_i>0$, while $Y_1,\ldots,Y_n$ 
are independent exponentially distributed random variables with hazard rates $\theta_j>0$. Let $I_1,I_2,I_3,I_4\subset\{1,\ldots,n\}$ be sets of integers and define $X_{(n_1)} = \max(X_{i}, i\in I_1)$, $X_{(n_2)} = \max(X_{j}, j\in I_2)$, $Y_{(m_1)} = \max(Y_{i}, i\in I_3)$, $Y_{(m_2)} = \max(X_{j}, j\in I_4)$, and $W=\min(X_{(n_1)},X_{(n_2)})$, $Z=\min(Y_{(m_1)},Y_{(m_2)})$. Let $N$ be the number of occurrences of $\lambda=\min(\lambda_i+\lambda_j,i\in I_1,j\in I_2)$ and $M$ be the number of occurrences of $\theta=\min(\theta_i+\theta_j,i\in I_3,j\in I_4)$. If $N=M$ then the variables $W$ and $Z$ are not comparable with respect to the convex transform order.
\end{proposition}
\begin{proof}
We have the following representations for the tail functions:
$$
\bF_W(x) =\left(1-\prod_{i\in I_1}(1-e^{-\lambda
_i x})\right) \left(1-\prod_{i\in I_2}(1-e^{-\lambda
_i x})\right)=Ne^{-\lambda x}+o(e^{-\lambda x})
$$
and
$$
\bF_Z(x) =\left(1-\prod_{i\in I_3}(1-e^{-\theta
_i x})\right) \left(1-\prod_{i\in I_4}(1-e^{-\theta
_i x})\right)=Me^{-\theta x}+o(e^{-\theta x}).
$$
Both tail functions are of class $\mathcal{ERPV}(-\infty)$, so, taking $c=\frac\lambda\theta$, it follows that
$$
\lim_{x\rightarrow+\infty}\frac{\bF_W(x)}{\bF_Z(cx)}=\lim_{x\rightarrow+\infty}\frac{Ne^{-\lambda x}}{Me^{-\theta cx}}=\frac{N}{M}=1,
$$
so, taking into account Theorems~\ref{NON-Convexity} and \ref{CONV_Characterization}, the conclusion follows.
\end{proof}

\begin{proposition}
	Let $X_1,\ldots, X_n$ be independent exponentially distributed random variables with hazard rates $0<\lambda_1\leq\cdots\leq\lambda_n$ and $Y_1,\ldots,Y_n$ be exponentially distributed random variables with hazard rates $0<\theta_1\leq\cdots\leq\theta_n$ such that their joint distribution is described by the F-G-M system
	\[
	F_{Y_1,\ldots,Y_n}(x_1,\ldots,x_n) = \prod_{i=1}^{n}(1-e^{-\theta_ix_i})\left(1+\sum_{1\leq i<j\leq n}c_{ij}e^{-(\theta_ix_i+\theta_jx_j)}\right)
	\]
where $\sum_{1\leq i<j\leq k}\vert c_{ij}\vert\leq 1$.
Assuming that the number of occurrences of $\lambda_1$ and $\theta_1$ is the same, then for any given integers $m,k\leq n$, the random variables $X_{k:k}$ and $Y_{m:m}$ are not comparable with respect to the convex transform order.
\end{proposition}
\begin{proof}
Observe that the distribution function $F_k(x)$ of $X_{k:k}$ is given by $F_k(x) = \prod_{i=1}^{k}(1-e^{-\lambda_i x})$, while the distribution function $G_m(x)$ of $Y_{m:m}$ is of the form $G_m(x) = \prod_{i=1}^{m}(1-e^{-\theta_ix})\left(1+\sum_{1\leq i<j\leq m}c_{ij}e^{-(\theta_i+\theta_j)x}\right)$. The corresponding survival functions can be written as
$\bF_k(x) = Ne^{-\lambda_1 x}+ o(e^{-\lambda_1 x})$ and $\bG_m(x) = Ne^{-\theta_1 x}+ o(e^{-\theta_1 x})$ and $N$ is the number of occurrences of $\lambda_1$ and $\theta_1$. Then, by choosing $c = \frac{\lambda_1}{\theta_1}$, we have that as $x\longrightarrow\infty$, the ratio $\frac{\bF_k(x)}{\bG_m(cx)}$ tends to 1, proving the non-comparability of the two random variables with respect to the convex transform order.
\end{proof}

Just for the sake of completeness, we also provide the result described above for the case of Weibull distributions.

\begin{proposition}
Let $X_1,\ldots, X_n$ be independent random variables with Weibull distributions with shape parameters $\alpha_1\leq \cdots\leq \alpha_n$ and hazard rates $\lambda_1,\ldots,\lambda_n$. Similarly, let $Y_1,\ldots,Y_n$ be Weibull distributed random variables with shape parameters $\beta_1\leq \cdots\leq \beta_n$ and hazard rates $\theta_1,\ldots,\theta_n$ such that their joint distribution is described by
	\[
	F_{Y_1,\ldots,Y_n}(x_1,\ldots,x_n) = \prod_{i=1}^{n}(1-e^{-\theta_i^{\beta_i}x_i^{\beta_i}})\left(1+\sum_{1\leq i<j\leq n}c_{ij}e^{-(\theta_i^{\beta_i}x_i^{\beta_i}+\theta_j^{\beta_j}x_j^{\beta_j})}\right)
	\]
where $\sum_{1\leq i<j\leq k}|c_{ij}|\leq 1$.
Assume that $\alpha_1 = \beta_1$ and the number of occurrences of $\alpha_1$ and $\beta_1$ is the same. Then for any given integers $m,k\leq n$ the random variables $X_{k:k}$ and $Y_{m:m}$ are not comparable with respect to convex transform order.	
\end{proposition}
\begin{proof}
Similarly to the proof of the previous result, $F_k(x) = \prod_{i=1}^{k}(1-e^{-\lambda_i^{\alpha_i}x^{\alpha_i}})$ and $G_{m}(x) = \prod_{i=1}^{m}(1-e^{-\theta_i^{\beta_i}x^{\beta_i}})\left(1+\sum_{1\leq i<j\leq n}c_{ij}e^{-(\theta_i^{\beta_i}x^{\beta_i}+\theta_j^{\beta_j}x^{\beta_j})}\right)$, while the corresponding survival functions are given by 	$\bF_k(x) = Ne^{-\lambda_1^{\alpha_1} x^{\alpha_1}}+o(e^{-\lambda_1^{\alpha_1} x^{\alpha_1}})$ and $\bG_m(x) = Ne^{-\theta_1^{\beta_1} x^{\beta_1}}+ o(e^{-\theta_1^{\beta_1} x^{\beta_1}})$ and $N$ is the number of occurrences of $\alpha_1$ and $\beta_1$. Recalling that $\alpha_1= \beta_1$, the conclusion follows in a similar manner as in the results described earlier.
\end{proof}

\end{document}